\documentclass{scrartcl}
\usepackage[utf8]{inputenc}
\usepackage[affil-it]{authblk} 

\usepackage{amsmath,amsthm,amsfonts,amssymb,amscd}
\usepackage{graphicx,bm,color}
\usepackage{algorithm}
\usepackage{subcaption}
\usepackage[T1]{fontenc}
\usepackage{algpseudocode}
\usepackage{stmaryrd}
\usepackage{listings}
\usepackage{lineno}
\usepackage{mathptmx}
\usepackage{cleveref}
\usepackage[toc,page]{appendix}
\usepackage[skins,theorems]{tcolorbox}
\usepackage{authblk}

\usetikzlibrary{arrows,calc, decorations.markings, intersections}
\usepackage{caption}

\usepackage{cancel}
\usepackage{mathtools}
\usepackage{url}
 
\def\keyword{\vspace{.5em}{\textbf{Keywords}.\,\relax}}

\definecolor{gray}{gray}{0.6}

\numberwithin{equation}{section}
\theoremstyle{plain}
\newtheorem*{theorem*}{Theorem}
\newtheorem{theorem}{Theorem}
\numberwithin{theorem}{section}
\newtheorem{proposition}[theorem]{Proposition}
\newtheorem{lemma}[theorem]{Lemma}

\theoremstyle{definition}
\newtheorem{definition}[theorem]{Definition}

\newcommand{\vertiii}[1]{{\left\vert\kern-0.25ex\left\vert\kern-0.25ex\left\vert #1 
		\right\vert\kern-0.25ex\right\vert\kern-0.25ex\right\vert}}

\renewcommand{\thefootnote}{\fnsymbol{footnote}}

\title{Multivariate Super-Resolution without Separation}
\author{Bakytzhan Kurmanbek}
\author{Elina Robeva}
\date{}
\affil{Department of Mathematics, The University of British Columbia, 1984 Mathematics Rd, Vancouver, BC V6T 1Z2}

\renewcommand{\cite}[1]{[#1]}
\def\beginrefs{\begin{list}%
        {[\arabic{equation}]}{\usecounter{equation}
         \setlength{\leftmargin}{2.0truecm}\setlength{\labelsep}{0.4truecm}%
         \setlength{\labelwidth}{1.6truecm}}}
\def\endrefs{\end{list}}

\usepackage{apacite}

\usepackage{xcolor}

\newcommand\blfootnote[1]{%
  \begingroup
  \renewcommand\thefootnote{}\footnote{#1}%
  \addtocounter{footnote}{-1}%
  \endgroup
}

\makeatletter
\def\blfootnote{\gdef\@thefnmark{}\@footnotetext}
\makeatother

\begin{document}
\maketitle
\begin{abstract}In this paper we study the high-dimensional super-resolution imaging problem. Here we are given an image of a number of point sources of light whose locations and intensities are unknown. The image is pixelized and is blurred by a known point-spread function arising from the imaging device. We encode the unknown point sources and their intensities via a nonnegative measure and we propose a convex optimization program to find it.  Assuming the device's point-spread function is componentwise decomposable, we show that the optimal solution is the true measure in the noiseless case, and it approximates the true measure well in the noisy case with respect to the generalized Wasserstein distance. Our main assumption is that the components of the point-spread function form a Tchebychev system ($T$-system) in the noiseless case and a $T^*$-system in the noisy case, mild conditions that are satisfied by Gaussian point-spread functions. Our work is a generalization to all dimensions of the work~\cite{eftekhari2021stable} where the same analysis is carried out in 2 dimensions. We resolve an open problem posed in~\cite{schiebinger2018superresolution} in the case when the point-spread function decomposes.
\end{abstract}
\blfootnote{\textup{2020} \textit{Mathematics Subject Classification}:
65K10, 15A30}
\keyword{Super-resolution, Tchebychev systems, Generalized-Wasserstein distance, Exact solutions, Bounds for a recovery error}

\section{Introduction}

In the super-resolution  imaging problem we are given the output of an imaging device depicting often very small or very distant objects. As a result, we observe a pixelized, blurred, and noisy image, and we aim to recover the true picture by removing the blur and increasing the resolution. Solving this problem is essential in many applied sciences such as neuroscience~\cite{betzig2006imaging, hess2006ultra, rust2006sub, ekanadham2011blind, evanko2009primer, tur2011innovation}, geophysics~\cite{khaidukov2004diffraction}, and astronomy~\cite{puschmann2005super}. Knowing the blurring (point-spread) function is usually crucial in achieving accurate results. 

Due to the complexity of the super-resolution imaging problem it is difficult to distinguish point sources of light that are very close to one another. As a result many of the existing methods can be proved to work only under a minimum separation condition between the point sources of light even as the amount of noise approaches zero~\cite{bendory2016robust, candes2014towards, morgenstern2015stable}.

It was first shown in the 1-dimensional case~\cite{schiebinger2018superresolution} that if no noise is present, then no minimum separation condition is required provided that the point-spread function satisfies a Tchebychev system~\cite{karlin1966tchebycheff} condition. This work was simplified and extended to the noisy case in~\cite{eftekhari2021sparse}, and then to the 2-dimensional and possibly noisy case in~\cite{eftekhari2021stable}. In the present paper we generalize these results to the case of any dimension both with or without noise.



A lot of the theoretical work on super-resolution imaging uses concepts and techniques from compressed sensing~\cite{donoho2006compressed, candes2006robust, candes2005decoding}, by minimizing total variation over measures~\cite{candes2014towards, tang2013compressed, fyhn2013spectral, demanet2013super, duval2015exact, denoyelle2017support, azais2015spike, bendory2016robust}. In this manuscript, assuming that we have $K$ source points in the exact image, we show that $(2K+1)^d$ observations are enough to recover the image. In the noiseless case, we recover the image exactly, and in the noisy case, we provide an error bound in Generalized-Wasserstein distance.

\subsection*{Mathematical setup}
Let $\mu$ be a Borel measure supported on $\mathbb{I}^{d} = [0, 1]^{d}$. This measure will represent the true image that we are trying to reconstruct. Given an image represented by the measure $\mu$, an imaging device with (a known) point-spread function $\Psi:\mathbb R^d \to \mathbb R$ convolves each individual point source of light with the point-spread function $\Psi$ resulting in a new image which, at a point $(x_1,\ldots, x_d)$ equals (up to a noise term)
$$\int_{\mathbb I^d}\Psi(x_1-t_1,\ldots, x_d-t_d)\mu[\text{d}(t_1,\ldots, t_d)].$$

We will later assume that $\mu$ is a {\em sparse} measure, i.e., $\mu = \sum_{k=1}^Ka_k\delta_{\theta_k}$, where $\theta_k = (t_1^{(k)}, \ldots, t_d^{(k)})\in\mathbb R^d$ are the point source locations, and $a_k > 0$ are their intensities. In this case, the observed image at a point $(x_1,\ldots, x_d)$ equals
$$\sum_{k=1}^Ka_k\Psi(x_1-t_1^{(k)},\ldots, x_d-t_d^{(k)}).$$

We will also assume that the point-spread function $\Psi$ lies in the {\em tensor product model}, i.e., 
\begin{equation}\label{eq::3}
    \Psi(\xi_1, \cdots, \xi_d) = \psi^{(1)}(\xi_1)\cdots\psi^{(d)}(\xi_d),
\end{equation}
where $\psi^{(1)},\ldots, \psi^{(d)}$ are continuous real-valued functions on $\mathbb I$. 

This is a widely used model in imaging as noted in~\cite{eftekhari2021stable} which studies the 2-dimensional case and assumes the 2-dimensional point-spread function satisfies~\eqref{eq::3}. It is also often the case that the functions $\psi^{(j)}$ are copies of a function $\psi$.

Suppose that the image we observe is of size $M\times \cdots \times M$ ($d$ times) with coordinates $\{y_{i_1, \cdots, i_d}\}_{i_j = 1}^{M}$ measured on an $M\times \cdots\times M$ ($d$ times) grid $\{x_1^{(1)},\ldots, x_M^{(1)}\}\times\cdots \times \{x_1^{(d)}, \ldots, x_M^{(d)}\}$, which can be thought of as consisting of the pixels. Then, 
\begin{equation}\label{eq::1}
    y_{i_1, \cdots, i_d} \approx \int_{\mathbb{I}^{d}} \psi^{(1)}(x_{i_1}^{(1)} - t_1) \cdots \psi^{(d)}(x_{i_d}^{(d)} - t_d) \mu [\text{d}(t_1, \cdots, t_d)],
\end{equation}
where we have noted the possibility of noise being present.  For notational convenience, we denote $$\psi^{(i)}_{m} (t_i) = \psi^{(i)}(x^{(i)}_{m_{i}} - t_{i})$$
for every $i \in [d]$ and $m \in [M]$. 
More specifically, we will assume that 
\begin{equation}\label{eq::2}
    \sum_{i_1, \cdots, i_d = 1}^{M} \Big| y_{i_1, \cdots, i_d} - \int_{\mathbb{I}^{d}} \psi^{(1)}_{i_1}(t_1) \cdots \psi^{(d)}_{i_d}(t_d) \mu [\text{d}(t_1, \cdots, t_d)] \Big|^{2} \leq \delta^{2}
\end{equation}
where $\delta \geq 0$ reflects an additive noise model. 

We may rewrite (\ref{eq::2}) more compactly as 
\begin{equation}\label{eq::4}
\Big\| y - \int_{\mathbb{I}^{d}} \psi^{(1)}_{[M]}(t_1)\otimes\cdots\otimes\psi^{(d)}_{[M]}(t_d) \mu [\text{d}(t_1, \cdots, t_d)] \Big\|_{F} \leq \delta
\end{equation}
where $\|\cdot\|_{F}$ stands for the Frobenious norm (of a tensor), and $\psi^{(j)}_{[M]}(t_j) = (\psi^{(j)}_1(t_j), \ldots, \psi^{(j)}_M(t_j))^T$.

\subsection*{Summary of results}
To recover $\mu$, we propose to use the convex {{feasibility program}}

    \begin{center}
    \textit{Find a nonnegative Borel measure $\mu$ on $\mathbb{I}^{d}$ such that
    \begin{equation}\label{eq::5}
        \Big\|y - \int_{\mathbb{I}^{d}} \psi^{(1)}_{[M]}(t_1)\otimes\cdots\otimes\psi^{(d)}_{[M]}(t_d)  \mu [\text{d}(t_1, \cdots, t_d)] \Big\|_{F} \leq \delta'
    \end{equation}}
\end{center}
for some $\delta' \geq \delta$. Note that this program is grid-free and has no regularity other than the non-negativity of the Borel measure compared to other methods where the measure's total variation is assumed to be constant~\cite{candes2014towards, tang2013compressed, fyhn2013spectral}. The dual program to the program (\ref{eq::5}) can be used to find the unknown point sources and their amplitudes. When there is no noise, i.e., $\delta' = 0$, we show that the solution of (\ref{eq::5}) is exact and unique under certain conditions for the component functions of the point spread function. In the presence of noise, i.e. $\delta > 0$, 
we show that the solution of program (\ref{eq::5})  well approximates the true measure, by showing an upper bound for the recovery error.

We generalize the main results obtained in~\cite{eftekhari2021stable} from 2 to higher dimensions. We address the \textit{{noiseless case}} in Section \ref{sec:2}. We show in Theorem \ref{theo:1} that if the true measure $\mu$ consists of $K$ point sources, then the unique solution of program (\ref{eq::5}) is the true measure $\mu$. Here we assume that the translates of the component-wise functions $\psi^{(j)}$ form Tchebychev systems and that $M \geq 2K+1$, where our observations $y \in (\mathbb{R}^{M})^{\bigotimes d}$. Moreover, no minimum separation between the  point sources is required, and, therefore, program (\ref{eq::5}) recovers correct solution no matter how close to each other the point sources may be. 

We address the \textit{{noisy case}} in Section~\ref{sec:3}. We assume the true measure is an arbitrary measure $\mu$ supported on $\mathbb{I}^{d}$ and the noise level is $\delta > 0$. Theorem \ref{theo:2} then shows that if the translates of the component-wise functions form $T^*$-systems, and given enough observations, the solution of program (\ref{eq::5}) well-approximates the true image measure in Generalized-Wasserstein distance. We use two auxiliary Lemmas \ref{lem:3} and \ref{lem:4}, which utilize the existence of dual certificates $Q$ and $Q^{0}$, to bound the errors away from the support and near the support. The existence of such dual certificates shown in Propositions \ref{prop:1} and \ref{prop:2} is a generalization to all dimensions of Propositions 21 and 22 in~\cite{eftekhari2021stable} which address this problem in 2 dimensions. We give an upper bound on the recovery error in terms of the noise level and the minimum separation between the points sources (see Theorem~\ref{theo:2}). 

The $T$-system and $T^{*}$-system assumptions have already been shown to hold for translated copies of Gaussian functions~\cite{eftekhari2021stable, eftekhari2021sparse}. Therefore, all of our results hold for Gaussian point spread functions that satisfy the decomposability criterion~\eqref{eq::3}. 

 The proofs of all theorems, lemmas, and propositions can be found in the appendices. 
\section{The noiseless case}\label{sec:2}
Let $\mu$ be a nonnegative atomic measure 
\begin{equation*}
    \mu = \sum_{k=1}^{K} a_{k} \delta_{\theta_k}, \;\;\; a_k >0
\end{equation*}
with $K \geq 1$ impulses located at $\Theta = \{\theta_k\}_{k=1}^{K} \subset \text{interior}(\mathbb{I}^{d})$ with positive amplitudes $\{a_k\}_{k=1}^{K}$, where $\theta_k = (t_1^{(k)}, \ldots, t_d^{(k)})\in\mathbb R^d$. Consider the case where there is no imaging noise ($\delta = 0$). In this case, we collect noise-free measurements 
\begin{equation*}
    y = \int_{\mathbb{I}^{d}} \psi^{(1)}_{[M]}(t_1)\cdots\psi^{(d)}_{[M]}( t_d)\mu[\text{d}(t_1, \cdots, t_d)] \in (\mathbb{R}^{M})^{\bigotimes d}.
\end{equation*}
To understand when solving program (\ref{eq::5}) with $\delta' = 0$ successfully recovers the true measure $\mu$, recall the concept of a T-system~\cite{karlin1966tchebycheff}:
\begin{definition}\label{def:T_system}
Several real-valued and continuous functions $\{\phi_j\}_{j=1}^{m}$ form a {\em Tchebychev system} (or \textit{T-system}) on the interval $\mathbb{I}$ if the $m\times m$ matrix $[\phi_j(\tau_i)]_{i, j= 1}^{m}$ is nonsingular for any increasing sequence $\{\tau_{i}\}_{i=1}^{m} \subset \mathbb{I}$. 
\end{definition}

Many sets of functions, such as the monomials $1,t,t^2,\ldots, t^{m-1}$ as well as shifted translates of Gaussian functions, are known to form $T$ systems. Given a $T$-system $\{\phi_m\}_{m=1}^{M}$, any linear combination of the elements $\sum_{i=1}^{m}b_{i} \phi_{i}$ is often called a "polynomial", and, like real polynomials of degree $m-1$, it has at most $m-1$ zeros on the real line. We will use some of the properties of $T$-systems to form a dual {\em polynomial} certificate for our program (\ref{eq::5}), which will be a linear combination of shiften copies of the point-spread function. The dual polynomial certificate will be the unique solution of the dual program to \eqref{eq::5} and will correspond to the unique solution to the original program \eqref{eq::5}. The construction of a dual certificate is a classical technique used in compressed sensing~\cite{candes2006robust, fuchs2005sparsity}.

\begin{theorem}[Noiseless measurements]\label{theo:1}
Let $\mu$ be a $K$-sparse nonnegative measure supported on interior$(\mathbb{I}^{d})$. Let also $M \geq 2K+1$, and suppose that the functions $\{\psi^{(i)}_{m_{i}}\}_{m_i=1}^{M}$ form a T-system on $\mathbb{I}$ for every $i=1,\ldots, d $. Lastly, consider the image $y \in (\mathbb{R}^{M})^{\bigotimes d}$ of the measure $\mu$ in \eqref{eq::4} with noise level $\delta = 0$. Then, $\mu$ is the unique solution of program \eqref{eq::5} with $\delta'=0$.
\end{theorem}
In other words, given at least $(2K+1)^{d}$ samples from an image consisting of $K$ point sources, program \eqref{eq::5} recovers the $K$ point sources exactly. Here, we require the component-wise functions $\{\psi^{(i)}_{m_{i}}\}_{m_i=1}^{M}$ to form a $T$-system  for every $i = 1,..., d$. Lemmas \ref{lem:2} and \ref{lem:2.1} below establish Theorem~\ref{theo:1}. Their proofs are located in Appendices \ref{apx:A} and \ref{apx:B}, respectively.

\begin{lemma}[Uniqueness of the measure]\label{lem:2}
Let $\mu$ be a $K$-sparse nonnegative atomic measure supported on $\Theta \subset interior(\mathbb{I}^{d})$. Then, $\mu$ is the unique solution of program \eqref{eq::5} with $\delta'=0$ if 
\begin{itemize}
    \item the $M^{d} \times K$ matrix $[\psi^{(1)}_{i_1}(t^{(k)}_{1}) \cdots \psi^{(d)}_{i_d}(t^{(k)}_{d})]_{i_1, \cdots, i_d, k=1}^{i_1, \cdots, i_d = M, k = K}$ has full rank, and 
    \item there exist real coefficients $\{b_{i_1, \cdots, i_d}\}_{i_1, \cdots, i_d = 1}^{M}$ and a polynomial $$Q(t_1, \cdots, t_d) = \sum_{i_1, \cdots, i_d =1}^{M}b_{i_1, \cdots, i_d} \psi^{(1)}_{i_1}(t_1)\cdots \psi^{(d)}_{i_d}(t_d)$$ such that $Q$ is nonnegative on $interior(\mathbb{I}^{d})$ and vanishes only on $\Theta$.
\end{itemize}
\end{lemma}
Our next lemma establishes the second one of the above conditions, namely the existence of a dual polynomial certificate for our program.
\begin{lemma}[Existence of a dual polynomial certificate]\label{lem:2.1}
Let $\mu$ be a $K$-sparse non-negative atomic measure supported on interior$(\mathbb{I}^d)$. For $M \geq 2K+1$, suppose that $\{\psi^{(i)}_{m_i}\}_{m_i = 1}^{M}$ form a T-system on $\mathbb{I}$ for $i = 1, ..., d$. Then the dual certificate $Q$ in Lemma \ref{lem:2} exists. 
\end{lemma}

Combining these two lemmas, we see that the true sparse measure $\mu$ is the unique solution to our optimization problem, which proves Theorem~\ref{theo:1}. 

\section{The noisy case}\label{sec:3}
In this section, we present our generalized results for the noisy case. 
Since the measure in this case, might not be atomic, we use atomic measures to approximate both true measure and the solution of program (\ref{eq::5}). We use the Generalized-Wasserstein distance for the recovery error. Before stating the main Theorem \ref{theo:2}, we introduce several necessary concepts.

\begin{definition}[Separation]\label{def:separation}
For an atomic measure $\mu$ supported on $\Theta = \{\theta_k\}_{k=1}^{K} = \{(t_{k}^{(1)}, \ldots, t_{k}^{(d)})\})_{k=1}^{K}$, let \textit{sep$(\mu)$} denote the \textit{minimum separation} between all impulses in $\mu$ and the boundary of $\mathbb{I}^{d}$. That is, $sep(\mu)$, is the largest number $v$ such that 
$$v \leq |t_{k}^{(i)} - t_{l}^{(i)}|, \; \forall\; k \neq l;\; k, l \in [K], i \in [d]$$
\begin{align*}
    v \leq |t_{k}^{(i)} - 0|, \;\;\; v  \leq |t_{k}^{(i)} - 1|\; \forall \;i \in [d], k \in [K].
\end{align*}
\end{definition}
We call a measure $\mu$, which satisfies $sep(\mu) = \epsilon$, an \textit{$\epsilon$-separated measure}. This means that all point sources are separated from each other and the boundary by at least $\epsilon$.

In Figure \ref{fig:my_label}, the separation of the atomic measure $\mu$ supported at two points (points $A$ and $B$) is  
\begin{align*}
    sep(\mu) = \min\{&t^{(1)}_{A}, t^{(2)}_{A}, t^{(3)}_{A}, t^{(1)}_{B}, t^{(2)}_{B}, t^{(3)}_{B}, \\ &1 - t^{(1)}_{A}, 1 - t^{(2)}_{A}, 1- t^{(3)}_{A}, 1 - t^{(1)}_{B}, 1 - t^{(2)}_{B}, 1- t^{(3)}_{B} \\
    &|t^{(1)}_{A} - t^{(1)}_{B}|, |t^{(2)}_{A} - t^{(2)}_{B}|, |t^{(3)}_{A} - t^{(3)}_{B}|\}
\end{align*}
\begin{figure}
    \centering
    \includegraphics[width=0.4\textwidth]{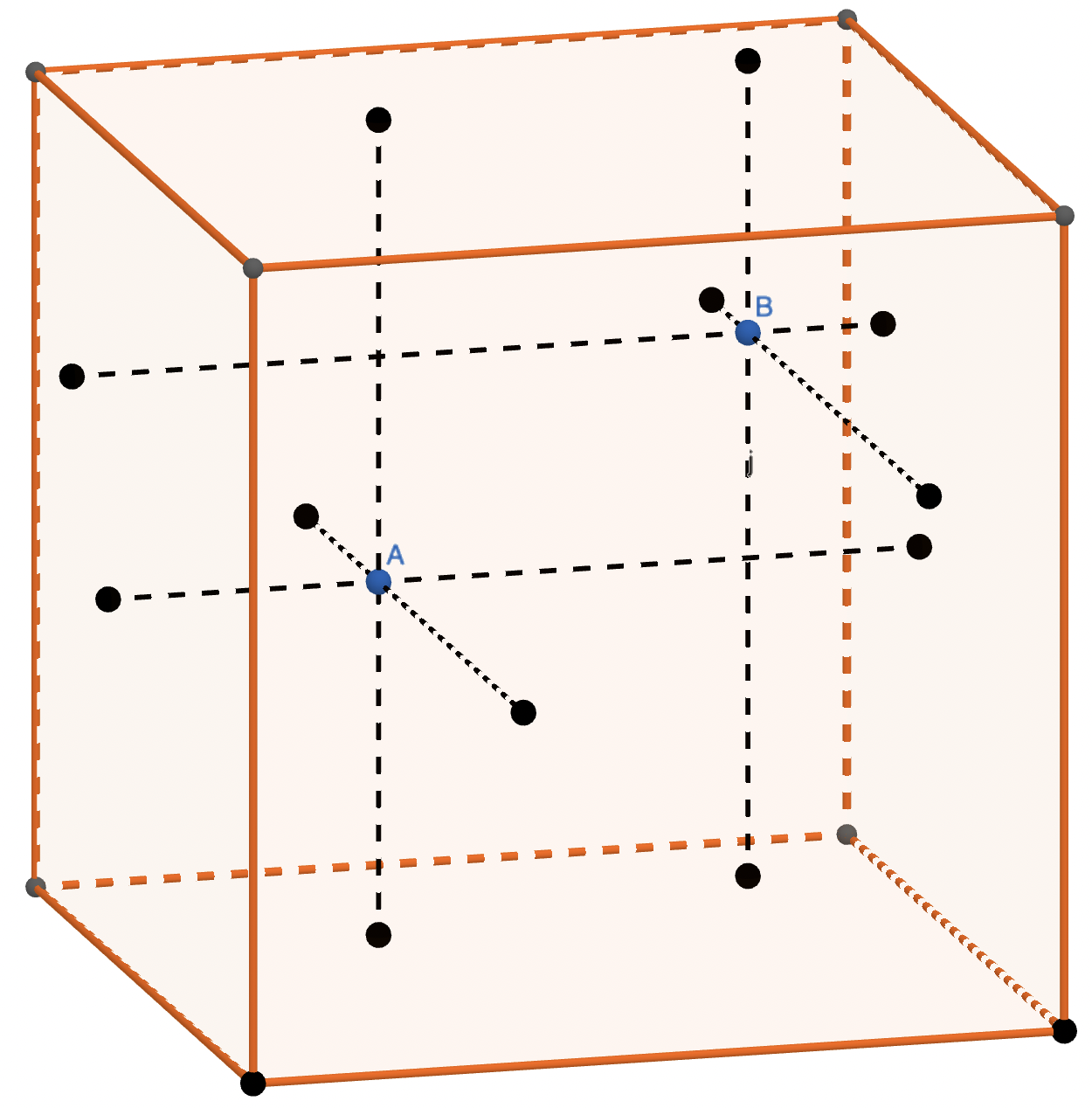}
    \caption{Separation in 3-D}
    \label{fig:my_label}
\end{figure}

Next, we introduce the Generalized-Wasserstein distance~\cite{piccoli2014generalized}, a natural notion of distance between measures that has been widely used in the optimal transport literature. We first recall the total variation (TV) norm and Wasserstein distance~\cite{villani2008optimal}. The total variation norm of a measure $\mu$ supported on $\mathbb{I}^{d}$ is $\|\mu\|_{TV} = \int_{\mathbb{I}^{d}}|\mu(\text{d}\theta)|$. It is similar to the $l_1$ norm for measures in finite dimensions. Wasserstein distance is defined for two non-negative measures $\mu_1$ and $\mu_2$, supported on $\mathbb{I}^{d}$, as 

\begin{equation}\label{eq:wasserstein}
    d_{W}(\mu_1, \mu_2) = \inf \int_{\mathbb{I}^{d} \times \mathbb{I}^{d}} \|\tau_1 - \tau_2\|_{1} \gamma(\text{d}\tau_1, \text{d}\tau_2)
\end{equation}
where the infimum is over every non-negative measure $\gamma$ on $\mathbb{I}^{d} \times \mathbb{I}^{d}$ that produces $\mu_1$ and $\mu_2$ as marginals, i.e., 
\begin{equation}\label{eq:marginals}
    \mu_1(A_1) = \int_{A_1 \times \mathbb{I}^d} \gamma(\text{d}\tau_1, \text{d}\tau_2), \;\;\;\;\mu_2(A_2) = \int_{\mathbb{I}^{d}\times A_2}\gamma(\text{d}\tau_1, \text{d}\tau_2) 
\end{equation}
for all measurable sets $A_1, A_2 \subseteq \mathbb{I}^d$. 
To find the Wasserstein  distance between $\mu_1$ and $\mu_2$, we require that the total variation norms of the two measures are equal, i.e., $\|\mu_1\|_{TV} = \|\mu_2\|$. However, this might might not be the case if $\mu_1$ is the true image measure and $\mu_2$ is the measure recovered by \eqref{eq::5}. Therefore, similar to~\cite{eftekhari2021stable}, we use the Generalized-Wasserstein distance between measures with unequal total variation. 

\begin{definition}[Generalized-Wasserstein distance]\label{def:GW}
For two non-negative measures $\mu_1$ and $\mu_2$, supported on $\mathbb{I}^{d}$, the \textit{Generalized-Wasserstein distance} between $\mu_1$ and $\mu_2$ is defined as 
\begin{equation}\label{eq:GW}
    d_{GW}(\mu_1, \mu_2) = \inf (\|\mu_1 - z_1\|_{TV} + d_{W}(z_1, z_2) + \|\mu_2 - z_2\|_{TV}).
\end{equation}
where the infimum is over pairs of non-negative Borel measures $z_1$ and $z_2$ supported on $\mathbb{I}^{d}$ such that $\|z_1\|_{TV} = \|z_2\|_{TV}$.
\end{definition}
We use the Generalized-Wasserstein distance $d_{GW}(\mu, \hat{\mu})$ to bound the distance between $\mu$ and $\hat\mu$ in Theorem~\ref{theo:2}, where $\mu$ is the true measure and $\hat{\mu}$ is the solution to the program (\ref{eq::5}). In bounding this quantity, we approximate the true measure by a sparse and well-separated atomic measure.  

\begin{definition}[Residual] \label{def:residual}
For a non-negative measure $\mu$ supported on $\mathbb{I}^{d}$, an integer $K$, and $\epsilon \in (0, 1/2]$, we define the residual
\begin{equation}\label{eq:residual}
    R(\mu, K, \epsilon) := d_{GW}(\mu, \mu_{K, \epsilon}) = \min d_{GW}(\mu, \nu) 
\end{equation}
where the minimum above is taken over all non-negative $K$-sparse and $\epsilon$-separated measures $\nu$ supported on interior($\mathbb{I}^{d}$). 
\end{definition}
In other words, $R(\mu, K, \epsilon)$ can be thought of as the \textit{mismatch} between the true measure $\mu$ and the closest well-separated sparse measure $\mu_{K,\varepsilon}$. The existence of the minimum in the above definition is explained in \cite{eftekhari2021stable}. 

In addition, we will impose a natural smoothness condition on the imaging apparatus, namely, L-Lipschitz-continuity.  
\begin{definition}[Smoothness]\label{def:smoothness}
The imaging apparatus $\Psi(\theta)$ is a \textit{L-Lipschitz-continuous} if 
\begin{equation}\label{eq:smoothness}
    \Big\|\int_{\mathbb{I}^{d}} \Psi(\theta) (\mu_1(\text{d}\theta) - \mu_2(\text{d}\theta)\Big\|_{F}\leq L \cdot d_{GW} (\mu_1, \mu_2)
\end{equation}
for every pair of measures $\mu_1, \mu_2$ supported on $\mathbb{I}^{d}$.
\end{definition}


Finally, following~\cite{eftekhari2021stable} and~\cite{eftekhari2021sparse} we also need to generalize the concept of a $T$-system so that we can use it in the noisy case. Firstly, we introduce  admissible sequences, which are sequences of ordered indexed numbers from $\mathbb{I}$ with specific properties. The limits of these numbers (except constants) approach at most $K$ different numbers in $\mathbb{I}$, and each limiting number has an even multiplicity, except for exactly one. 

\begin{definition}[Admissible sequence]\label{def:admissible_sequence}
For a pair of integers $K$ and $M$ obeying $M \geq 2K + 1$, we say that $\{\{\tau_{k}^{n}\}_{k=0}^{M}\}_{n \geq 1} \subset \mathbb{I}$ is a \textit{$(K, \epsilon)$-admissible sequence} if: 
\begin{itemize}
    \item $\tau^{n}_{0} = 0$ and $\tau_{M}^{n} = 1$ for every $n$
    \item As $n \to \infty$, the increasing sequence $\{\{\tau_{k}^{n}\}_{k=1}^{M-1}\}$ converges (element-wise) to an $\epsilon$-separated finite subset of $\mathbb{I}$ with at most $K$ distinct points, where every element has an even multiplicity, except one element that appears only once. 
\end{itemize}
\end{definition}
\begin{figure}
    \centering
    \includegraphics[width=0.7\textwidth]{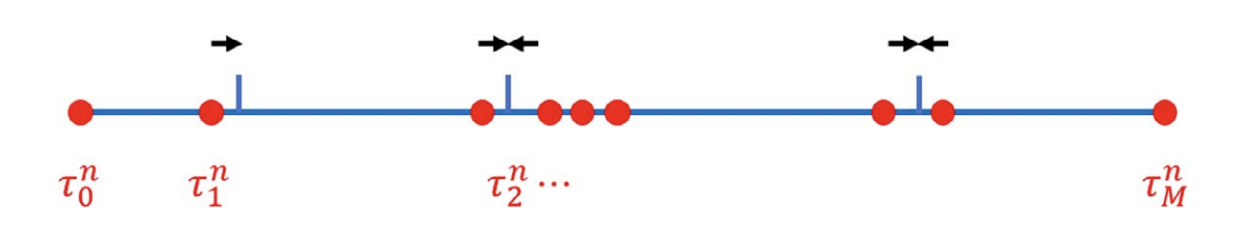}
    \caption{Figure from "Stable super-resolution of images: theoretical study" by \protect\citeauthor{eftekhari2021stable}, 2021, \protect\textit{Information and Inference: A Journal of the IMA, 10} (1), p. 169}
    \label{fig:sequence}
\end{figure}
Recall that  $T$-systems are defined for all increasing sequences in $\mathbb{I}$. In contrast,  $T^*$-systems are defined for all admissible sequences, requiring that the determinant of the matrix evaluated at subsequences of an admissible sequence is positive and all of the minors of this matrix along a row converge to zero at the same rate.

\begin{definition}[$T^*$-system]\label{def:T*_system}
For an integer $K$ and an even integer $M$ obeying $M \geq 2K + 2$, the real-valued functions $\{\phi_{m}\}_{m = 1}^{M}$ form a $T^{*}_{K, \epsilon}$-system if for every $(K, \epsilon)$-admissible sequence $\{\{\tau_{k}^{n}\}_{k=0}^{M}\}_{n \geq 1}$:
\begin{itemize}
    \item The determinant of the $(M+1) \times (M+1)$ matrix $[\phi_{m}(\tau^{n}_{k})]_{k, m = 0}^{M}$ is positive for all sufficiently large~$n$. 
    \item All minors along the $l$th row of the above matrix $[\phi_{m}(\tau^{n}_{k})]_{k, m = 0}^{M}$ approach zero at the same rate when $n \to \infty$. Here, $l$ is the index of the element of the limit sequence that appears only once.
\end{itemize}
\end{definition}

\smallskip
We now mention the following important properties of $T^*$-systems:
\begin{itemize}
    \item A $T^*_{K, \epsilon}$-system on $\mathbb{I}$ is also a $T^{*}_{K', \epsilon}$-system for every integer $K' \leq K$, because every $(K', \epsilon)$-admissible sequence is also a $(K, \epsilon)$-admissible sequence. 
    \item If $\{\phi_{m}\}_{m=1}^{M}$ form a $T^{*}_{K, \epsilon}$-system on $\mathbb{I}$, then so do the scaled functions $\{c_m \phi_m\}_{m=1}^{M}$ for any positive constants $\{c_m\}_{m=1}^{M}$. 
\end{itemize}

In the following theorem, we introduce our main result, which generalizes Theorem 11 from ~\cite{eftekhari2021stable} to higher dimensions.  Theorem \ref{theo:2} gives an upper bound for the recovery error in terms of the noise level $\delta \geq 0$ and the separation $\epsilon \in (0, 1/2]$. 

\begin{theorem}[Noisy measurements]\label{theo:2}
Consider a non-negative measure $\mu$ supported on $\mathbb{I}^{d}$, a noise level $\delta \geq 0$, and the image $y \in (\mathbb{R}^{M})^{\bigotimes d}$ obtained from $\mu$ as in~\eqref{eq::1} and \eqref{eq::4}. Assume further that the imaging apparatus is $L$-Lipschitz continuous in the sense of (\ref{eq:smoothness}). 

For an integer $K$ and $\epsilon \in (0, 1/2]$, let $\mu_{K, \epsilon}$ be a $K$-sparse and $\epsilon$-separated non-negative measure supported on $\mathbb{I}^{d}$ that approximates $\mu$ with residual $R(\mu, K, \epsilon)$. Let $\Theta = \{\theta_k\}_{k=1}^{K} = \{(t^{(1)}_{k}, ..., t^{(d)}_{k} )\}_{k=1}^{K} \subset \mathbb{I}^{d}$ denote the support of $\mu_{K, \epsilon}$, and set $T_{\Omega_i} := \{t^{(i)}_{k} \}_{k \in \Omega_i \subseteq [K]}$ for each $i\in [d]$. 

If $\hat{\mu}$ is the solution of program (\ref{eq::5}), with $\delta' \geq (1+L \cdot R(\mu, K, \epsilon))\delta$, then 
\begin{equation}\label{noise:bound}
    d_{GW}(\mu, \hat{\mu}) \leq c_1 \delta + c_2(\epsilon) + c_3 R(\mu, K, \epsilon)
\end{equation}
where $d_{GW}$ is a Generalized-Wasserstein metric (\ref{eq:wasserstein}), above $c_1, c_2(\epsilon), c_3$ are specified in the proof, and they depend on the measure $\mu$, the separation $\epsilon$ and the measurement functions $\{\psi_m\}_{m=1}^{M}$.

The error bound (\ref{noise:bound}) holds if $M \geq 2K+2$ and
\begin{enumerate}
    \item $\{\psi^{(i)}_{m_i}\}_{m_i = 1}^{M}$ forms a T-system on $\mathbb{I}$ for $i \in [d]$, \\
    \item $\{F_{T_{\Omega_i}}\} \cup \{\psi^{(i)}_{m_i}\}_{m_i = 1}^{M}$ forms a $T^*$-system on $\mathbb{I}$ for every $\Omega_i \subseteq [K]$ and $i \in [d]$, \\
    \item $\{F^{\pm}_{t^{(i)}_{k}}\} \cup \{\psi^{(i)}_{m_i}\}_{m_i = 1}^{M}$ forms a $T^*$-system on $\mathbb{I}$ for every $k \in [K]$ and $i \in [d-1]$, \\
    \item $\{F^{+}_{t^{(d)}_{k}}\} \cup \{\psi^{(i)}_{m_i}\}_{m_i = 1}^{M}$ forms a $T^*$-system on $\mathbb{I}$ for every $k \in [K]$ and $i\in [d]$,
\end{enumerate}
where the functions $F_{T_{\Omega_i}}, F_{t_k^{(d)}}^+$, and $F_{t_k^{(d)}}^{\pm}$ are defined below.
\end{theorem}
For a finite set of distinct points $T' \subset \mathbb{I}$ and $\epsilon \in (0, sep(T')]$, we define the function $F_{T'}: \mathbb{I} \rightarrow \mathbb{R}$ as
\begin{equation}\label{eq_n::1}
    F_{T'}(t) := \begin{cases}
    0, \;\text{when there exists} \;t' \in T'\; \text{such that}\; t \in t'_{\epsilon} \\
    1, \;\text{elsewhere on}\; \mathbb{I},
    \end{cases}
\end{equation}
where $t'_{\epsilon} = \{t \in \mathbb{I}: |t - t'| \leq \epsilon\}$ is the $\epsilon$-neighborhood of $t'$. 
For $t' \in \mathbb{I}$ and $\epsilon \in (0, 1/2]$, we define the functions $F^{\pm}_{t'}, F^{+}_{t'}: \mathbb{I} \rightarrow \mathbb{R}$ as
\begin{equation*}
    F^{\pm}_{t'}(t) := \begin{cases}
    \pm 1, \;\text{when}\; t\in t'_{\epsilon} \\
    0, \;\text{elsewhere on}\; \mathbb{I}
    \end{cases}\;\;\text{and}\;\;\;\;F^{+}_{t'}(t) := \begin{cases}
    1, \;\text{when}\; t\in t'_{\epsilon} \\
    0, \;\text{elsewhere on}\; \mathbb{I}.
    \end{cases}
\end{equation*}

\textit{\textbf{Outline of the proof of Theorem \ref{theo:2}.}} In the proof of Theorem \ref{theo:2}, we use Lemmas \ref{lem:3} and \ref{lem:4} and Propositions \ref{prop:0}, \ref{prop:1} and \ref{prop:2}. The first two lemmas are concerned with bounding the error $h = \hat{\mu} - \mu_{K, \epsilon}$ around and away from the support, where $\hat{\mu}$ and $\mu_{K, \epsilon}$ are the solution of the program (\ref{eq::5}) and the approximation of the true measure of the image, respectively. The above mentioned lemmas require the existence of particular dual certificates. Propositions \ref{prop:1} and \ref{prop:2} guarantee the existence of such dual certificates under some of the conditions required in Theorem \ref{theo:2}. The results presented in Propositions \ref{prop:1} and \ref{prop:2} are generalizaitons to the high-dimensional case of their 2-dimensional versions in~\cite{eftekhari2021stable}. Proposition \ref{prop:0} combines the results from Lemmas \ref{lem:3} and \ref{lem:4}, and finds the error bound for $d_{GW}(\mu_{K, \epsilon}, \hat{\mu})$. 

We use the following notation throughout the paper. For a positive $\epsilon$ and $T^{(i)} = \{t_{k}^{(i)}\}_{k=1}^{K} \subset \mathbb{I}$, let us define the neighborhoods
\begin{align*}
    t_{k, \epsilon}^{(i)} = \{t \in \mathbb{I}: \; |t - t^{(i)}_{k}| \leq \epsilon\} \subset \mathbb{I}, 
    \end{align*}
$$T^{(i)}_{\epsilon} = \bigcup_{k=1}^{K} t^{(i)}_{k, \epsilon}.$$
Also, let 
$$\theta_{k, \epsilon}:= t^{(1)}_{k, \epsilon} \times ... \times t^{(d)}_{k, \epsilon} = \{\theta \in \mathbb{I}^{d}: \|\theta - \theta_{k}\|_{\infty} \leq \epsilon\} \subset \mathbb{I}^{d},$$
$$\Theta_{\epsilon} := \bigcup_{k=1}^{K}\theta_{k, \epsilon} \subseteq T^{(1)}_{\epsilon} \times ... \times T^{(d)}_{\epsilon}.$$

\begin{lemma}[Error away from the support]\label{lem:3}
Let $\hat{\mu}$ be the solution to program~\eqref{eq::5} with $\delta' \geq \delta$ and set $h: = \hat{\mu} - \mu_{K, \epsilon}$ to be the error. Fix a positive scalar $\bar{g}$. Suppose that there exist real coefficients $\{b_{m_1, \cdots , m_d}\}_{m_1, \cdots , m_d = 1}^{M}$ and a polynomial 
$$Q(\theta) = Q(t^{(1)}, \cdots , t^{(d)}) = \sum_{m_1, \cdots , m_d = 1}^{M} b_{m_1, \cdots , m_d} \psi^{(1)}_{m_1}(t^{(1)}) \cdots \psi^{(d)}_{m_d}(t^{(d)})$$
such that 
\begin{equation}\label{eq:er_away}Q(\theta) \geq G(\theta) = 
\begin{cases}
0, \; \text{when there exists}\; k \in [K]\;\text{such that}\;\theta \in \theta_{k, \epsilon}, \\
\bar{g}, \;\text{elsewhere in} \; interior(\mathbb{I}^{d}),
\end{cases}
\end{equation}
where equality holds when $\theta\in\Theta = \{\theta_{k}\}_{k=1}^{K}$. Then,
$$\int_{\Theta_{\epsilon}^{C}} h(d\theta) \leq 2 \|b\|_{F} \delta'/\bar{g},$$
where $b \in (\mathbb{R}^M)^{\bigotimes d}$ is the tensor with entries $\{b_{m_1, \cdots, m_d}\}_{m_1, \cdots , m_d = 1}^{M}$.
\end{lemma}

Assuming a dual certificate $Q$ exists, Lemma \ref{lem:3} bounds the error away from the support. 
 The proof is stated in Appendix \ref{apx:C}, and uses the triangle and Cauchy-Schwartz inequalities. In the next lemma, we estimate the error near the support. 

\begin{lemma}[Error near the support]\label{lem:4}
Let $\hat{\mu}$ be the solution to program~\eqref{eq::5} with $\delta' \geq \delta$ and set $h: = \hat{\mu} - \mu_{K, \epsilon}$ to be the error. For $\alpha \in [0, 1]$, suppose that there exist real coefficients $\{b^{0}_{m_1, \cdots, m_d}\}_{m_1, \cdots, m_d = 1}^{M}$ and a polynomial
$$Q^{0}(\theta) = Q^{0}(t^{(1)}, \cdots , t^{(d)}) = \sum_{m_1, \cdots , m_d = 1}^{M} b^{0}_{m_1, \cdots , m_d} \psi^{(1)}_{m_1}(t^{(1)}) \cdots \psi^{(d)}_{m_d}(t^{(d)})$$
such that 
\begin{equation}\label{eq:0.2}
    Q^{0}(\theta) \geq G^{0}(\theta): = 
    \begin{cases}
    1, &\text{when there exists}\; k\in [K]\; \text{such that}\;\theta \in \theta_{k, \epsilon}\; \text{and}\; \int_{\theta_{k, \epsilon}}h(d\theta) >0 \\
    -1, &\text{when there exists}\; k\in [K]\; \text{such that}\;\theta \in \theta_{k, \epsilon}\; \text{and}\; \int_{\theta_{k, \epsilon}}h(d\theta) \leq 0 \\
    -1+\alpha, &\text{when there exists}\; k\in [K]\; \text{such that}\;\theta = \theta_{k, \epsilon}\; \text{and}\; \int_{\theta_{k, \epsilon}}h(d\theta) \leq 0 \\
    -1, &\text{elsewhere in} \; interior(\mathbb{I}^{d})
    \end{cases}
\end{equation}
where equality holds when $\theta\in\Theta$. Then, 
\begin{equation}
    \sum_{k=1}^{K} \Big|\int_{\theta_{k, \epsilon}} h(d\theta) \Big| \leq \alpha\|\mu_{K, \epsilon}\|_{TV} + 2\Big(\|b^{0}\|_{F} + \frac{\|b\|_{F}}{\bar{g}} \Big)\delta'
\end{equation}
where $b^{0} \in (\mathbb{R}^M)^{\bigotimes d}$ is the tensor with entries  $\{b^{0}_{m_1, \cdots, m_d}\}_{m_1, \cdots , m_d = 1}^{M}$.
\end{lemma}

Assuming a dual certificate $Q^{0}$ exists, Lemma \ref{lem:4} bounds the error near the support, and together with Lemma \ref{lem:3}, it is used to bound the Generalized-Wasserstein distance between the solution of program (\ref{eq::5}) and the approximation $\mu_{K,\varepsilon}$ of the exact measure with a $K$-sparse and $\epsilon$-separated atomic measure supported on $\mathbb{I}^{d}$.

If the dual certificates $Q$ and $Q^0$ exist, then inequality (\ref{eq:wass_metric}) below holds. The proof in the 1-dimensional case is presented in Lemma 18 of ~\cite{eftekhari2021sparse}, and we here extended it to higher dimensions.
\begin{proposition}[Error in Wasserstein metric] \label{prop:0}
Suppose the dual certificates $Q$ and $Q^0$ in Lemmas (\ref{lem:3}) and (\ref{lem:4}) exist. Then, 
\begin{equation}\label{eq:wass_metric}
    d_{GW}(\mu_{K, \epsilon}, \hat{\mu}) \leq \Big( \frac{8\|b\|_{F}}{\bar{g}} + 6\|b^{0}\|_{F} \Big)\delta' + (\epsilon+3\alpha)\|\mu_{K, \epsilon}\|_{TV}.
\end{equation}
\end{proposition}
In the next two propositions, we show the existence of dual certificates $Q$ and $Q^0$ under some mild conditions. These results generalize similar propositions from ~\cite{eftekhari2021sparse, eftekhari2021stable} to the higher-dimensional case. The detailed proofs of Propositions \ref{prop:1} and \ref{prop:2} can be found in Appendices \ref{apx:E} and \ref{apx:F}, respectively.

\begin{proposition}[Dual certificate for error away from support] \label{prop:1}
Suppose that $\{\psi^{(i)}_{m_i}\}_{m_i=1}^{M}$ form a T-system on $\mathbb{I}$ for $i \in [d]$ with $M \geq 2K+2$. For every partition $\pi = \{\Omega_1, \ldots, \Omega_d\}$ of $[K]$, suppose also that all $\{F^{(i)}_{T_{\Omega_i}}\} \cup \{\psi^{(i)}_{m}\}_{m=1}^{M}$ 
are $T^*_{K, \epsilon}$-systems on $\mathbb{I}$ for each $i\in [d]$. Then the dual certificate $Q$, as specified in Lemma \ref{lem:3}, exists with 
$$\bar{g} = (d-1) d^{K-2}.$$
\end{proposition}

\begin{proposition}[Dual certificate for error near support]\label{prop:2}
For $M \geq 2K+2$, suppose that $\{\psi^{(i)}_{m}\}_{m=1}^{M}$ form a T-system on $\mathbb{I}$. 
Suppose also that all $\{F^{\pm}_{t^{(j)}_{k}}\} \cup \{\psi^{(i)}_{m}\}_{m=1}^{M}$ 
are $T^*_{K, \epsilon}$-systems on $\mathbb{I}$ for each $1 \leq j \leq d-1$. Additionally, $\{F^{+}_{t^{(d)}_{k}}\} \cup \{\psi^{(i)}_{m}\}_{m=1}^{M}$ also form $T^*_{K, \epsilon}$-systems on $\mathbb{I}$. Then the dual certificate $Q^{0}$, as specified in Lemma \ref{lem:4} , exists with 
$$\alpha = \alpha(\epsilon) = 1 + (-1)^{d-1}/q^{\pi_{*}}_{max}(\epsilon),$$
where $q^{\pi_{*}}_{max}(\epsilon)$ is defined in the proof.
\end{proposition}

\section{Discussion}
In this paper we have shown that program~\eqref{eq::5} solves the super-resolution imaging problem in any dimension, assuming that the point-spread function factorizes, and that translates of each of the factors satisfy a number of $T$-system and $T^*$-system conditions. Our work extends~\cite{eftekhari2021stable} from the 2-dimensional to the $d$-dimensional case. We show that if there is no measurement noise, then, one can recover the true locations of the point-sources regardless of how close they are to each other, and, if there is noise, then the quality of the estimate 
will depend on the separation of the point-sources through the noise level.

We wish to remark that here we show that the output of a certain convex program~\eqref{eq::5} is the desired true image, but we do not provide an algorithm for solving this program. In practice, as noted in~\cite{eftekhari2021sparse}, one may want to minimize a certain objective function, such as the total variation of the unknown measure $\mu$. There already exist algorithms for solving such programs, for example, see~\cite{boyd2017}.

Of course, one of the main open questions that remain is whether our results can extend to the case when the point-spread function does not factorize. Solving this problem may require developing a high-dimensional analog of $T$-systems, which are currently only defined for functions in one variable. Perhaps a simpler start is to consider point-spread functions that factorize after a suitable rotation is applied -- this is at least true for all Gaussian point-spread functions.

\section*{Acknowledgements}
We wish to thank Geoffrey Schiebinger and Jean-Baptiste Seby for helpful discussions at the beginning stages of this project. ER was supported by an NSERC Discovery Grant (DGECR-2020-00338).
\bibliographystyle{apacite}
\bibliography{reference}
\appendix
\section{Proof of Lemma \ref{lem:2}} \label{apx:A}
Clearly, the true measure $\mu = \sum_{k=1}^{K}a_{k} \delta_{\theta_k}$ is a solution of \eqref{eq::5}. Suppose that there is another measure $\hat{\mu}$ that solves \eqref{eq::5}. Then, by \eqref{eq::5} we necessarily have 
$$\int_{\mathbb{I}^{d}} \Psi_{i_1, \cdots, i_d} \text{d}\mu = \int_{\mathbb{I}^{d}} \Psi_{i_1, \cdots, i_d} \text{d}\hat{\mu}$$
for all $i_1, ..., i_d \in [M]$. Therefore, summing both sides over $i_1, ..., i_d$ with coefficients $b_{i_1, ..., i_d}$, we get that 
$$\int_{\mathbb{I}^{d}}Q\text{d}\mu = \int_{\mathbb{I}^{d}} Q\text{d}\hat{\mu}$$
But note that $Q = 0$ on $\Theta$, therefore, 
$$\int_{\mathbb{I}^{d}\setminus \Theta} Q\text{d}\mu = \int_{\mathbb{I}^{d} \setminus \Theta} Q\text{d}\hat{\mu}$$
But $\mu = 0$ on $\mathbb{I}^{d}\setminus \Theta$, and $Q > 0$ on $\mathbb{I}^{d} \setminus\Theta$. Therefore, $\hat{\mu} = 0$ on $\mathbb{I}^{d} \setminus\Theta$.

\section{Proof of Lemma \ref{lem:2.1}}\label{apx:B}
The following lemma has been stated in~\cite{eftekhari2021sparse} in Lemma 15, and proved in~\cite{karlin1966tchebycheff} in Theorem 5.1. 
\begin{lemma}[Univariate polynomial of a T-system] \label{lem:uni}
Consider a set $T' \subset \mathbb{I}$ of size $K'$. With $M \geq 2K'+1$, suppose that $\{\phi_{m}\}_{m=1}^{M}$ form a T-system on $\mathbb{I}$. Then, there exist coefficients $\{b_m\}_{m=1}^{M}$ such that the polynomial $q_{T'} = \sum_{m=1}^{M}b_m \phi_m$ is non-negative on $\mathbb{I}$ and vanishes only on $T'$. 
\end{lemma}

Consider any partition $\pi$ of $[K]$ into $d$ parts, $\Omega_1, ..., \Omega_d \subset [K]$, and we denote throughout of this paper \begin{equation}\label{eq:partition}
    T_{\Omega_1} = \{t^{(1)}_{k}\}_{k \in \Omega_1}, \cdots, T_{\Omega_d} = \{t^{(d)}_{k}\}_{k \in \Omega_d}.
\end{equation} By assumption $\{\psi^{(i)}_{m_i}\}_{m_i=1}^{M}$ forms a T-system for each $i = 1, ..., d$, so by Lemma \ref{lem:uni}, there exist polynomials $q_{T_{\Omega_1}}, \cdots, q_{T_{\Omega_d}}$ that are nonnegative on $\mathbb{I}$ and vanish only on $T_{\Omega_1}, \cdots, T_{\Omega_d}$ respectively. 

Let us form a polynomial 
$$Q(\theta) = Q(t^{(1)}, \cdots, t^{(d)}) = \sum_{\text{all partitions}\; \pi \;\text{of}\; [K]}q_{T_{\Omega_1}}(t^{(1)})\cdot ... \cdot q_{T_{\Omega_d}}(t^{(d)}),$$
where the sum is all over partitions of $[K]$. Obviously, $Q$ is a nonnegative on $\mathbb{I}^{d}$, because each summand is nonnegative. Now, let's verify $Q$ vanishes at $\Theta$. Take $\theta_{k} = (t^{(1)}_{k}, \cdots, t^{(d)}_{k})$, then since $k$ always belongs to one of partitions, each summand would be zero at $\theta_k$.  

When $\theta \in \Theta^{C}$, then consider 
\begin{enumerate}
    \item For $\theta = (t^{(1)}, ..., t^{(d)}) \in (T^{(1)}_{\epsilon})^{C} \times \cdots \times (T^{(d)}_{\epsilon})^{C} \subseteq \Theta^{C}_{\epsilon}$, it holds 
    $$q_{T_{\Omega_1}}(t^{(1)})\cdot ...\cdot q_{T_{\Omega_d}}(t^{(d)}) > 0$$
    for every partition, by summing up over all partitions, it immediately follows 
    $$Q(\theta) >  0$$
    \item For $\theta \in \Theta^{C}_{\epsilon} \setminus \Big((T^{(1)}_{\epsilon})^{C} \times \cdots \times (T^{(d)}_{\epsilon})^{C}\Big)$. We will consider two cases
    \begin{enumerate}
        \item Since the order of indexes doesn't matter, assume that there exist $k_1, ..., k_l \in [K]$, might all equal to each other with $l<d$, such that $\theta \in t^{(1)}_{k_1, \epsilon} \times \cdots \times t^{(l)}_{k_l, \epsilon} \times t^{(l+1)} \times \cdots t^{(d)}$. There exist a partition $\pi^*$ such that the first $l$ elements of it $\Omega^{*}_1, ..., \Omega^{*}_l$ doesn't contain any of these numbers $\{k_1, ..., k_l\}$, so for this partition 
        $$q_{T_{\Omega^{*}_1}}(t^{(1)})\cdot ...\cdot q_{T_{\Omega^{*}_d}}(t^{(d)})>0.$$
        Therefore, 
        $$Q(\theta) > 0$$
       \item The case when $l=d$, which means there exist $k_1, ..., k_d \in [K]$, not all equal to each other, such that $\theta \in t^{(1)}_{k_1, \epsilon} \times \cdots \times t^{(d)}_{k_d, \epsilon}$. Consider a partition $\pi^*$ such that $\Omega^*_{i}$ doesn't contain $k_{i}$, so 
       $$q_{T_{\Omega^{*}_1}}(t^{(1)})\cdot ...\cdot q_{T_{\Omega^{*}_d}}(t^{(d)})>0$$
       And, we have $$Q(\theta) > 0.$$
    \end{enumerate}
\end{enumerate}
In conclusion, $Q(\theta)$ is everywhere positive except it vanishes only on $\Theta$.

\section{Proof of Lemma \ref{lem:3}}\label{apx:C}
By feasibility of both $\hat{\mu}$ and $\mu_{K, \epsilon}$ for the Program, applying the triangle inequality, we get 
$$\Big\|\int_{\mathbb{I}^{d}} \Psi(\theta) h(\text{d}\theta)\Big\|_{F} \leq 2 \delta'$$
Now, using the existence of the dual certificate
\begin{align*}
    \bar{g}\int_{\Theta^{C}_{\epsilon}} h(\text{d}\theta) &\leq \int_{\Theta^{C}_{\epsilon}} G(\theta) h(\text{d}\theta) \;\; (\text{since}\; h \; \text{is nonnegative on} \; \Theta^{C}_{\epsilon}) \\
    &= \int_{\Theta^{C}_{\epsilon}} G(\theta) h(\text{d}\theta) + \sum_{k=1}^{K}\int_{\theta_{k, \epsilon}} G(\theta) h(\text{d}\theta) \\
    &= \int_{\mathbb{I}^{d}} G(\theta) h(\text{d}\theta)\\
    &\leq \sum_{m_1, \cdots , m_d = 1}^{M} b_{m_1, \cdots , m_d} \int_{\mathbb{I}^{d}}\psi^{(1)}_{m_1}(t^{(1)}) \cdots \psi^{(d)}_{m_d}(t^{(d)})h(\text{d}t^{(1)}, \cdots, \text{d}t^{(d)})\;\; \text{(by \eqref{eq:er_away})} \\
    &= \Big\langle b, \int_{\mathbb{I}^{d}}\Psi(\theta) h(\text{d}\theta) \Big\rangle\\
    &\leq \|b\|_{F} \Big\|\int_{\mathbb{I}^{d}} \Psi(\theta) h(\text{d}\theta)\Big\|_{F} \;\;(\text{Cauchy-Schwarz inequality}) \\
    &\leq 2 \|b\|_{F} \delta'
\end{align*}
where $b \in (\mathbb{R}^M)^{\bigotimes d}$.

\section{Proof of Lemma \ref{lem:4}}\label{apx:D}

The existence of the dual certificate $Q^{0}$ allows us to write that 
\begin{align*}
    \sum_{k=1}^{K} \Big|\int_{\theta_{k, \epsilon}} h(\text{d}\theta) \Big| &=\\
    &= \sum_{k=1}^{K} \int_{\theta_{k, \epsilon}} s_k h(\text{d}\theta) \;\;\;\;\Big(s_k:=sign\Big(\int_{\theta_{k, \epsilon}} h(\text{d}\theta)\Big)\Big) \\
    &=\sum_{k=1}^{K} \int_{\theta_{k, \epsilon}} (s_k - Q^{0}(\theta)) h(\text{d}\theta) + \sum_{k=1}^{K} \int_{\theta_{k, \epsilon}} Q^{0}(\theta) h(\text{d}\theta)\\
    &=\sum_{k=1}^{K} \int_{\theta_{k, \epsilon}} (s_k - Q^{0}(\theta)) h(\text{d}\theta) + \int_{\mathbb{I}^{d}} Q^{0}(\theta) h(\text{d}\theta) - \int_{\Theta^{C}_{\epsilon}} Q^{0}(\theta) h(\text{d}\theta)\\
    &=\sum_{s_k=1} \int_{\theta_{k, \epsilon}} (1 - Q^{0}(\theta)) h(\text{d}\theta)+ \sum_{s_k=-1} \int_{\theta_{k, \epsilon}} (-1 - Q^{0}(\theta)) h(\text{d}\theta) \\
    &\;\;\;\;\;\;\;\;\;\;+\int_{\mathbb{I}^{d}} Q^{0}(\theta) h(\text{d}\theta) - \int_{\Theta^{C}_{\epsilon}} Q^{0}(\theta) h(\text{d}\theta) \\
    &= \sum_{s_k=1} \int_{\theta_{k, \epsilon}} (1 - Q^{0}(\theta)) h(\text{d}\theta)+ \sum_{s_k=-1} \int_{\theta_{k, \epsilon} \setminus \theta_{k}} (-1 - Q^{0}(\theta))h(\text{d}\theta) \\
    &+ \sum_{s_k=-1} \int_{\theta_{k}} (-1 + \alpha - Q^{0}(\theta)) h(\text{d}\theta) - \alpha \sum_{s_k=-1} \int_{\theta_{k}} h(\text{d}\theta)
     + \int_{\mathbb{I}^{d}} Q^{0}(\theta) h(\text{d}\theta) - \int_{\Theta^{C}_{\epsilon}} Q^{0}(\theta) h(\text{d}\theta)\\
     &\leq - \alpha \sum_{s_k=-1} \int_{\theta_{k}} h(\text{d}\theta)
     + \int_{\mathbb{I}^{d}} Q^{0}(\theta) h(\text{d}\theta) + \int_{\Theta^{C}_{\epsilon}}  h(\text{d}\theta) \;\;\;\; (\text{by \eqref{eq:0.2} })\\
     &\leq - \alpha \sum_{s_k=-1} \int_{\theta_{k}} h(\text{d}\theta)
     + \int_{\mathbb{I}^{d}} Q^{0}(\theta) h(\text{d}\theta) + 2 \|b\|_{F}  \delta'/\bar{g}  \;\;\;\;(\text{by previous lemma}) \\
     &\leq \alpha \sum_{s_k=-1} \int_{\theta_{k}} \mu_{K, \epsilon}(\text{d}\theta)
     + \int_{\mathbb{I}^{d}} Q^{0}(\theta) h(\text{d}\theta) + 2 \|b\|_{F}  \delta'/\bar{g} \;\;\;\;(\text{since}\; h: =\hat{\mu} - \mu_{K, \epsilon}\; \text{and}\; \hat{\mu} \geq 0) \\
     &=\alpha \sum_{s_k=-1} \sum_{k=1}^{K}\int_{\theta_{k}} a_{k}(\text{d}\theta)
     + \int_{\mathbb{I}^{d}} Q^{0}(\theta) h(\text{d}\theta) + 2 \|b\|_{F}  \delta'/\bar{g} \;\;\;\;(\text{since}\; \mu_{K, \epsilon} = \sum_{k=1}^{K} a_{k} \delta_{\theta_{k}})\\
     &\leq \alpha \|\mu_{K, \epsilon}\|_{TV}
     + \int_{\mathbb{I}^{d}} Q^{0}(\theta) h(\text{d}\theta) + 2 \|b\|_{F}  \delta'/\bar{g} \\
     &=\alpha \|\mu_{K, \epsilon}\|_{TV}
     + \Big\langle b^{0}, \int_{\mathbb{I}^{d}}\Psi(\theta) h(\text{d}\theta) \Big\rangle + 2 \|b\|_{F}  \delta'/\bar{g}\\
     &\leq \alpha \|\mu_{K, \epsilon}\|_{TV}
     + \| b^{0}\|_{F} \cdot 2\delta' + 2 \|b\|_{F}  \delta'/\bar{g} \;\;\;\;(\text{by Cauchy-Schwarz and ineq. from prev. lemma})
\end{align*}
Above the tensor $b^{0} \in (\mathbb{R}^M)^{\bigotimes d}$ formed by coefficients of dual certificate $Q^{0}$.


\section{Proof of Proposition \ref{prop:1}}\label{apx:E}

In the proof of  proposition we will use the lemma about $T^*$-systems; see~\cite{eftekhari2021sparse}, Proposition 19 for proof details.  
\begin{lemma}[Univariate polynomial of a $T^*$-system] \label{lem:5}
Consider a finite set $T' \subset \mathbb{I}$ of size no longer than $K$. For $M \geq 2K+2$, suppose that $\{\phi_{m}\}_{m=1}^{M}$ form a T-system on $\mathbb{I}$. Consider also $F': \mathbb{R} \rightarrow \mathbb{R}$ and suppose that $\{F'\} \cup \{\phi_{m}\}_{m=1}^{M}$ form a $T^*_{K, \epsilon}$-system on $\mathbb{I}$. Then there exist real coefficients $\{b_{m}\}_{m=1}^{M}$ and a continuous polynomial $q_{T'} = \sum_{m=1}^{M}b_{m} \phi_{m}$ such that $q_{T'} \geq F'$ with equality holding on $T'$. 
\end{lemma}

Consider any partition $\pi = \{\Omega_1, ..., \Omega_d\}$ of $[K]$ and define $T_{\Omega_i}$ in the sense of \eqref{eq:partition}. By assumption, $\{F_{T_{\Omega_i}}\} \cup \{\psi^{(i)}_{m_i}\}_{m_i=1}^{M}$ form a $T^{*}_{K, \epsilon}$-system on $\mathbb{I}$ for each $i \in [d]$. Therefore, by above lemma, there exist polynomials $q_{T_{\Omega_i}}$ for each $i \in [d]$ such that 
$$q_{T_{\Omega_i}} \geq F_{T_{\Omega_i}}$$
with equality holding on $T_{\Omega_i}$. 

Consider the polynomial 
\begin{equation*}
    Q(\theta) = Q(t^{(1)}, \cdots, t^{(d)}) = \sum_{all\;partitions\; \pi \;of [K]} q_{T_{\Omega_1}}(t^{(1)})\cdots q_{T_{\Omega_d}}(t^{(d)})
\end{equation*}
where the sum is over all partitions of $[K]$. We next show that $Q$ is the desired dual certificate. 

For every $\theta \in \theta_{k, \epsilon}$, it holds that 
\begin{align*}
    Q(\theta) = \sum_{all\;partitions\; \pi \;of [K]} q_{T_{\Omega_1}}(t^{(1)})\cdots q_{T_{\Omega_d}}(t^{(d)}) \geq \sum_{all\;partitions\; \pi \;of [K]} F_{T_{\Omega_1}}(t^{(1)})\cdots F_{T_{\Omega_d}}(t^{(d)}) = 0
\end{align*}
because we know for sure that every partition has a set $\Omega_j$ which contains $k$, that means $F_{T_{\Omega_j}}(t^{(j)}) = 0$. And, equality achieves exactly at $\Theta$.

On the other hand, consider $\theta \in \Theta^{C}_{\epsilon}$. 
\begin{enumerate}
    \item For $\theta = (t^{(1)}, ..., t^{(d)}) \in (T^{(1)}_{\epsilon})^{C} \times \cdots \times (T^{(d)}_{\epsilon})^{C} \subseteq \Theta^{C}_{\epsilon}$, it holds 
    $$q_{T_{\Omega_1}}(t^{(1)})\cdots q_{T_{\Omega_d}}(t^{(d)}) \geq 1$$
    for every partition, by summing up over all partitions, it immediately follows 
    $$Q(\theta) \geq d^{K}$$
    \item For $\theta \in \Theta^{C}_{\epsilon} \setminus \Big((T^{(1)}_{\epsilon})^{C} \times \cdots \times (T^{(d)}_{\epsilon})^{C}\Big)$. We will consider two cases
    \begin{enumerate}
        \item Since the order of indexes doesn't matter, assume that there exist $k_1, ..., k_l \in [K]$, might all equal to each other with $l<d$, such that $\theta \in t^{(1)}_{k_1, \epsilon} \times \cdots \times t^{(l)}_{k_l, \epsilon} \times t^{(l+1)} \times \cdots t^{(d)}$. Suppose among these $\{k_1, ..., k_l\}$ indexes, there are only $p$ different numbers, and WLOG, we can take
        \begin{align*}
            &k_1 = \cdots = k_{a_1} \\
            &k_{a_1+1} = \cdots = k_{a_1+a_2} \\
            &\;\;\;\;\;\vdots \\
            &k_{a_1+\cdots + a_{p-1}+1} = \cdots = k_{a_1+\cdots + a_p}
        \end{align*}
        with $a_1+\cdots+ a_p = l$, where each $a_i \geq 1$. Using this notation, we can write 
        \begin{align*}
            Q(\theta) = &\sum_{\pi} q_{T_{\Omega_1}}(t^{(1)})\cdot ... \cdot q_{T_{\Omega_{a_1}}}(t^{(a_1)}) \cdot q_{T_{\Omega_{a_1+1}}}(t^{(a_1+1)}) \cdot ... \cdot q_{T_{\Omega_{a_1+a_2}}}(t^{(a_1+a_2)})\cdot  \\
            &... \cdot q_{T_{\Omega_{a_1+\cdots+a_{p-1}+1}}}(t^{(a_1+\cdots+a_{p-1}+1)}) \cdot ... \cdot q_{T_{\Omega_{a_1+\cdots+a_{p}}}}(t^{(a_1+\cdots+a_{p})}) \cdot q_{T_{\Omega_{l+1}}}(t^{(l+1)}) \cdot ... \cdot q_{T_{\Omega_d}}(t^{(d)})
        \end{align*}
        This polynomial is not zero only when the first $a_1$ components of $\theta$ are not in $\Omega_1, ..., \Omega_{a_1}$, and the second $a_2$ components are not in $\Omega_{a_1+1}, ..., \Omega_{a_1+a_2}$, and so on, then all the components starting from $t^{(l+1)}$ can be in any place, therefore 
        $$Q(\theta) \geq (d-a_1)\cdot \cdots \cdot(d-a_p) d^{K-p}$$
        with $a_1+...+a_p = l$, $l\geq 1, \; a_i \geq 1, \; p\geq 1$. 
        
        Using Lemma \ref{lem:6}, we know that for each $1\leq l \leq d-1$, we have $Q(\theta) \geq (d-l)d^{K-1}$, and it achieves its minimum when $l=d-1$, so 
        $$Q(\theta) \geq d^{K-1}$$
        \item The case when $l=d$, which means there exist $k_1, ..., k_d \in [K]$, not all equal to each other, such that $\theta \in t^{(1)}_{k_1, \epsilon} \times \cdots \times t^{(d)}_{k_d, \epsilon}$. Repeating the same steps as above, we obtain a minimization problem 
        $$Q(\theta) \geq (d-a_1)\cdot \cdots \cdot (d-a_p)d^{K-p}$$
        with $a_1+\cdots+a_p = d$, where $a_i\geq 1$ and $p\geq 2$. \\
         When $p=2$, consider a function $P(x) = (d-x)(d-(d-x)) = (d-x)x$, which is concave down, and achieves it's minimum at endpoints $x = 1$, $P(1) = d-1$, so $Q(\theta) \geq (d-1)d^{K-2}$\\
    So for $l=d$, using Lemma \ref{lem:7} 
    $$Q(\theta) \geq (d-1)d^{K-2}$$
    \end{enumerate} 
 Comparing the lower bounds in all above cases, we see that 
    $$Q(\theta) \geq (d-1)d^{K-2}$$
    for all $d$.     
\end{enumerate}
\begin{lemma}\label{lem:6}
    For a fixed $l$, with $l < d$, the following 
    $$(d-l) d^{p-1} \leq (d-a_1)\cdot ... \cdot (d-a_p)$$
    is true for any partition of $l$, i.e., $a_1+\cdots+a_p = l$, $l\geq 1, \; p\geq 1, \; a_i \geq 1$. 
\end{lemma}
\begin{proof}
We proceed by induction on $l$. The base case is trivial, $l=1$, we have $p=1$, so $(d-1) \leq (d-1)$. Suppose we proved for $l=n-1$ for some $n$, so we have $$(d-n+1)d^{p-1} \leq (d-a_1)\cdot ... \cdot (d-a_p)$$ for any partition of $n-1$, i.e., $a_1+\cdots+a_p = n-1, \; p\geq 1, \; a_i\geq 1$. \\
Now, let's prove for $l=n$ case, and consider any partition of $n$, say $b_1+\cdots+b_q = n$, then $\{b_1, ..., b_q\}$ belongs to one of these forms 
\begin{itemize}
    \item If there exists $1 \in \{b_1, ..., b_q\}$, then WLOG, $b_1 + \cdots + b_{q-1} =n-1$ is a partition of $n-1$, so 
    $$(d-n)d^{q-1} \leq (d-n+1)(d-1)d^{q-2} \leq (d-b_1)\cdot... \cdot (d-b_{q-1}) (d-1) = (d-b_1)\cdot ... \cdot (d-b_q)$$
    \item If $1 \notin \{b_1, ..., b_q\}$, then $(b_1 -1) + \cdots + b_q = n-1$ is a partition of $n-1$
    \begin{align*}
    (d-n)d^{q-1} &= (d-n+1)d^{q-1}-d^{q-1} \leq (d-b_1+1)\cdot ... \cdot (d-b_q) - d^{q-1} = \\
    &= (d-b_1)\cdot ... \cdot (d-b_q)+ (d-b_2)\cdot ... \cdot (d-b_q)- d^{q-1} \leq (d-b_1)\cdot ... \cdot (d-b_q)
    \end{align*}
        due to the last difference is negative. 
\end{itemize}
This finishes our proof.
\end{proof}
\begin{lemma}\label{lem:7}
For any partition of $d$ such that $a_1+\cdots +a_p =d$, where $p \geq 3$ and $a_i \geq 1$, the following is satisfied $$(d-a_1)\cdot ... \cdot (d-a_p) d^{K-p}\geq (d-1)d^{K-2}$$
\end{lemma}
\begin{proof}
The above inequality is equivalent to 
$$(d-a_1)\cdot ... \cdot (d-a_p) \geq (d-1)d^{p-2}$$
Let's notice that for any positive $a, b \geq 1$, we have $(d-a)(d-b) \geq d(d-(a+b))$, and applying it many times, we get 
\begin{align*}
    (d-a_1)\cdot ... \cdot (d-a_p) &\geq d (d-(a_1+a_2))(d-a_3)\cdot ... \cdot (d-a_p) \geq \\
     &... \geq d^{p-2}(d-(a_1+\cdots+a_{p-1}))(d-a_p) = d^{p-2}a_{p}(d-a_p) \geq (d-1)d^{p-2}
\end{align*}
    The last inequality came from $p=2$ case. 
\end{proof}
\section{Proof of Proposition \ref{prop:2}}\label{apx:F}

The proof is based on Lemma \ref{lem:5}, let us now fix an arbitrary sign pattern $\pi \in \{\pm1\}_{k=1}^{K}$. For every $k \in [K]$, by assumption $\{F^{+}_{t^{(d)}_{k}}\} \cup \{\psi^{(i)}_{m_i}\}_{m_i=1}^{M}$ form a $T^*$-system, therefore, there exists a polynomial $q^{\pi_k}_{t^{(d)}_{k}}$ such that 
\begin{align*}
    q^{\pi_k}_{t^{(d)}_{k}} \geq \begin{cases}
    F^{+}_{t^{(d)}_{k}} \;\text{when}\; \pi_k = 1 \\
    \epsilon^{d-1}F^{+}_{t^{(d)}_{k}}\;\text{when}\; \pi_k = -1
    \end{cases}
\end{align*}
for every $t^{(d)}_{k} \in \mathbb{I}$, with equality holding on $T^{(d)} = \{t^{(d)}_{k}\}_{k=1}^{K}$. When the sign pattern $\pi$ contains at least one negative
sign, we define for future use the normalized maximum
$$q^{\pi}_{\max}(\epsilon): = \epsilon^{-(d-1)}\max_{\pi_k = -1} \max_{t^{(d)} \in \mathbb{I}}q^{\pi_k}_{t^{(d)}_{k}}$$
where the inner maximum above is indeed achieved in view of the continuity of $q^{\pi_k}_{t^{(d)}_{k}}$ and the compactness of $\mathbb{I}$. Note that $q^{\pi}_{\max}(\epsilon) \geq 1$ for every $\epsilon > 0$. When $\epsilon = 0$, we choose the trivial polynomial $q^{\pi_k}_{t^{(d)}_{k}} = \epsilon = 0$ for every $k$ such that $\pi_k = -1$, and thus let's take it as $$q^{\pi}_{\max}(0) = 1.$$
Likewise, for every $k \in [K]$, $\{F^{\pm}_{t^{(j)}_{k}}\} \cup \{\psi^{(j)}_{m_j}\}_{m_j=1}^{M}$ form a $T^*$-system on $\mathbb{I}$ by assumption. Therefore, for every $k\in [K]$ there exists a polynomial $q^{\pi_{k}}_{t^{(j)}_{k}}$ such that 
\begin{align*}
    q^{\pi_{k}}_{t^{(j)}_{k}} \geq \begin{cases}
    F^{\pi_k}_{t^{(j)}_{k}}\;&\text{when}\; \pi_k = 1 \\
    \frac{F^{\pi_k}_{t^{(j)}_{k}}}{\epsilon q^{\pi}_{\max}(\epsilon)}\;&\text{when}\; \pi_k = -1
    \end{cases}
\end{align*}
for every $t^{(j)}_{k}$ with equality holding on $\{t^{(j)}_{k}\}_{k=1}^{K}$. Now, we observe that 
\begin{align*}
    q^{\pi_{k}}_{t^{(1)}_{k}} \cdot ... \cdot q^{\pi_{k}}_{t^{(d-1)}_{k}} \cdot q^{\pi_{k}}_{t^{(d)}_{k}} \geq \begin{cases}
    \pi_k \;&\text{when}\; \theta \in \theta_{k, \epsilon}\\
    \frac{(-1)^{d-1}}{q^{\pi}_{\max}(\epsilon)}\;&\text{when}\; \theta = \theta_{k}\; \text{and}\;\pi_k = -1 \\
    -1\;&\text{elsewhere in}\; \mathbb{I}^{d}
    \end{cases}
\end{align*}
with equality holding at least on $\Theta$. Let us now consider the polynomial
\begin{align*}
    Q^{\pi}(\theta) := \sum_{k \in [K]} q^{\pi_k}_{t^{(1)}_{k}}(t^{(1)})\cdot ... \cdot q^{\pi_k}_{t^{(d-1)}_{k}}(t^{(d-1)}) q_{t^{(d)}_{k}}(t^{(d)})
\end{align*}
We establish that for the appropriate choice of the sign pattern $\pi$, $Q^{\pi}$ is indeed the desired dual certificate prescribed in Lemma \ref{lem:4}.More specifically, let $\pi^*$ denote the sign pattern specified by the error measure $h$ in (\ref{eq:0.2}), then the dual certificate exists with 
$$\alpha(\epsilon) = 1 + \frac{(-1)^{d-1}}{q^{\pi^*}_{\max}(\epsilon)}$$
This completes the proof.

\section{Proof of Theorem \ref{theo:2}}
Using the $L$-Lipschitz property of the imaging apparatus and using the triangle inequality, we can write
\begin{align*}\label{eq:mis}
    \Big\|y - \int_{\mathbb{I}^{d}} \Psi(\theta) \mu_{K, \epsilon}(\text{d}\theta) \Big\|_{F} &\leq 
    \Big\|y - \int_{\mathbb{I}^{d}} \Psi(\theta) \mu(\text{d}\theta) \Big\|_{F} + \Big\| \int_{\mathbb{I}^{d}} \Psi(\theta) (\mu (\text{d}\theta)-  \mu_{K, \epsilon}(\text{d}\theta)) \Big\|_{F} \\ 
    &\leq \delta + L \cdot d_{GW}(\mu, \mu_{K, \epsilon}) \\
    &= \delta + L \cdot R(\mu, K, \epsilon) := \delta'
\end{align*}
We can think that $\hat{\mu}$ is a solution of rogram (\ref{eq::5}) with $\delta'$ as an estimation of $\mu_{K, \epsilon}$. And since we constructed appropriate dual certificates in Propositions \ref{prop:1} and \ref{prop:2}, we can use the results to bound the the generalized Wasserstein metric error between the true measure and the solution of program (\ref{eq::5}). 
\begin{align*}
    d_{GW}(\mu, \hat{\mu}) &\leq d_{GW}(\mu, \mu_{K, \epsilon}) + d_{GW}(\mu_{K, \epsilon}, \hat{\mu}) \;\;(\text{triangle inequality})\\
    &\leq R(\mu, K, \epsilon) + \Big(\frac{8\|b\|_{F}}{\bar{g}} + 6\|b^0\|_{F} \Big)\delta' + (\epsilon + 3\alpha(\epsilon))\|\mu_{K, \epsilon}\|_{TV}\;\;\text{(Proposition \ref{prop:0})} \\
    &=  R(\mu, K, \epsilon) + \Big(\frac{8\|b\|_{F}}{\bar{g}} + 6\|b^0\|_{F} \Big)(\delta + L \cdot R(\mu, K, \epsilon)) + (\epsilon + 3\alpha(\epsilon))\|\mu_{K, \epsilon}\|_{TV}\\
    &= \Big(\frac{8\|b\|_{F}}{\bar{g}} + 6\|b^0\|_{F} \Big)\delta + \Big(\frac{8L\|b\|_{F}}{\bar{g}} + 6L\|b^0\|_{F} + 1 \Big)R(\mu, K, \epsilon)+ (\epsilon + 3\alpha(\epsilon))\|\mu_{K, \epsilon}\|_{TV}\\
    &= \Big(\frac{8}{d-1}d^{2-K} \|b\|_{F}+ 6\|b^0\|_{F} \Big)\delta + \Big(\frac{8L}{d-1}d^{2-K} \|b\|_{F} + 6L\|b^0\|_{F} + 1 \Big)R(\mu, K, \epsilon) \\&\;\;\;\;+ (\epsilon  +3\alpha(\epsilon))\|\mu_{K, \epsilon}\|_{TV}
\end{align*}
Next, we relate $\|\mu_{K, \epsilon}\|_{TV}$ to $\|\mu\|_{TV}$. So, 
\begin{align*}
    \|\mu_{K, \epsilon}\|_{TV} &= d_{GW}(\mu_{K \epsilon}, 0)\\
    &\leq d_{GW}(\mu_{K, \epsilon}, \mu) + d_{GW}(\mu, 0)\;\;\text{triangle inequality}\\
    &=  d_{GW}(\mu_{K, \epsilon}, \mu) + \|\mu\|_{TV} \\
    &\leq d_{GW}(\mu, 0) + \|\mu\|_{TV} \\
    &= 2\|\mu\|_{TV}
\end{align*}

So, combining the two above results, we get 
\begin{align*}
    c_1 &= \frac{8}{d-1}d^{2-K} \|b\|_{F}+ 6\|b^0\|_{F}\\
    c_2(\epsilon) &= (2\epsilon + 6\alpha(\epsilon))\|\mu\|_{TV}\\
    c_3 &=\frac{8L}{d-1}d^{2-K} \|b\|_{F} + 6L\|b^0\|_{F} + 1
\end{align*}

\end{document}